\documentclass[12pt]{amsart}
\usepackage{amsfonts,amssymb,amscd,amsmath,amsthm,enumerate,verbatim,calc,}
\usepackage[all]{xy}

\theoremstyle{plain}

\newcommand{\bb}{\begin{tiny} \begin{bmatrix}}
\newcommand{\eb}{\end{bmatrix} \end{tiny}}

\newtheorem{Theorem}{Theorem}[section]

\newtheorem{main theorem}{Main Theorem}
\newtheorem{Lemma}[Theorem]{Lemma}
\newtheorem{Corollary}[Theorem]{Corollary} 
\newtheorem{Proposition}[Theorem]{Proposition}

\newtheorem{Conjecture}{Conjecture}
\theoremstyle{definition}
\newtheorem{Definition}[Theorem]{Definition}
\newtheorem{Example}[Theorem]{Example}
\newtheorem{Remark}[Theorem]{Remark}

\newtheorem{Discussion}[Theorem]{Discussion}
\theoremstyle{remark}

\newtheorem*{chunk*}{}

\numberwithin{equation}{Theorem}

\newcommand{\mlabel}[1]%
  {\mbox{}\marginpar{\raggedleft\hspace{0pt}{\rm\ttfamily#1}}\label{#1}}

\newcommand{\Mod}{\operatorname{Mod}}

\newcommand{\fm}{{\mathfrak m}}

\newcounter{hours}\newcounter{minutes}

\newcommand{\excise}[1]{}

\begin{document}
\title[Almost Cohen-Macaulay algebras in mixed characteristic via Fontaine rings]
{Almost Cohen-macaulay algebras in mixed characteristic via Fontaine rings}
\author[K.Shimomoto]{Kazuma Shimomoto}

\address{Department of Mathematics, School of Science and Technology, Meiji University, 1-1-1 Higashimita, Tama-ku, Kawasaki 214-8571, Japan.}
\email{shimomotokazuma@gmail.com}




\begin{abstract} 
In the present paper, it is proved that any complete local domain of mixed characteristic has a weakly almost Cohen-Macaulay algebra $B$ in the sense that a system of parameters is a weakly almost regular sequence in $B$, which is a notion defined via a valuation. In fact, the central idea of this result originates from the main statement obtained by Heitmann to prove the Monomial Conjecture in dimension 3. A weakly almost Cohen-Macaulay algebra is constructed over the absolute integral closure of a complete local domain by applying the methods of Fontaine rings and Witt vectors. A connection of the main theorem with the Monomial Conjecture is also discussed.
\end{abstract}

\maketitle

\section{Introduction}

Let $(R,\fm)$ be a local Noetherian ring with a system of parameters $x_{1},\ldots,x_{d}$. We recall that an $R$-algebra $B$ is called a \textit{big Cohen-Macaulay R-algebra} if $\fm B \ne B$ and the sequence $x_{1},\ldots,x_{d}$ is $B$-regular. There is no finiteness condition on $B$. The following conjecture raised by Hochster~\cite{Ho75} has been of central interest in the study of certain homological conjectures.

\begin{Conjecture}[Hochster]
Every local Noetherian ring of mixed characteristic has a big Cohen-Macaulay algebra.
\end{Conjecture}

Big Cohen-Macaulay algebras are known to exist for equicharacteristic local rings (see~\cite{Ho75},~\cite{HH2}). In fact, more is true in the equicharacteristic case. That is, the existence of weakly functorial big Cohen-Macaulay algebras was established by Hochster and Huneke based upon their main theorem of~\cite{HH2} together with the reduction of the characteristic zero case to the positive characteristic case via Artin approximation theorem.

Suppose that $R$ is a local domain of mixed characteristic and $\dim R \le 3$. Then it was shown by Hochster~\cite{Ho02} that $R$ has a big Cohen-Macaulay algebra by examining Heitmann's proof of the Monomial Conjecture in dimension 3~\cite{Hei}. Therefore, the conjecture remains open for local rings in mixed characteristic of dimension at least 4. There has not been much progress on the existence problem of big Cohen-Macaulay algebras for several years. However, an approach has been recently found by Roberts~\cite{Ro1},~\cite{Ro2} aiming at finding an almost Cohen-Macaulay algebra, whose existence suffices to prove the Monomial Conjecture in general, and he proves under certain conditions, that almost Cohen-Macaulay algebras exist in mixed characteristic. The importance of almost Cohen-Macaulay algebras in the sense we need was first recognized in Heitmann's theorem quoted above. We also mention that there is an extensive study on closure operations of ideals of Noetherian rings defined by big Cohen-Macaulay algebras~\cite{Die}.

In this article, we attempt to shed some light on the conjecture of Hochster in the mixed characteristic case from a different perspective. First, we make a definition.

Let $(R,\fm)$ be a complete local domain. Then there is a discrete valuation $v:R \to \mathbb{Z} \cup \{\infty \}$ such that $v$ is positive on $R$ and strictly positive on $\fm$. Let $R^{+}$ be the integral closure of $R$ in an algebraic closure of the field of fractions of $R$. Then from a general theory on valuations, there is a valuation $v_{R^{+}}:R^{+} \to \mathbb{Q} \cup \{\infty\}$, which is an extension of $v$ from $R$ to $R^{+}$. For simplicity, we denote this extended valuation $v_{R^{+}}$ by $v$.

\begin{Definition}
Let the notation be as above, let $B$ be an $R^{+}$-algebra, and let $x_1,\ldots,x_d$ be a system of parameters for $R$. Then we say that $x_1,\ldots,x_d$ is \textit{weakly almost regular} on $B$, if $\fm B \ne B$ and for any rational $\epsilon>0$, there exists an element $b \in R^{+}$ such that $v(b)<\epsilon$ and 
$$
b \cdot \frac{(x_1,\ldots,x_i)B:_{B}x_{i+1}}{(x_1,\ldots,x_i)B}=0
$$
for all $0 \le i \le d-1$. An $R^{+}$-algebra $B$ is called \textit{weakly almost Cohen-Macaulay}, if there is a system of parameters for $R$ that is weakly almost regular on $B$.
\end{Definition}

We may choose any fixed valuation on $R^{+}$, as long as it proves desired results. As we noted above, the notion of almost Cohen-Macaulay algebras is suggested by Roberts \cite{Ro2}, where he further assumes that $B/\fm B$ is not almost zero in the sense that every element of $B/\fm B$ is not annihilated by elements of $R^{+}$ with arbitrarily small valuations (see~\cite{Ro2} for a precise definition of almost zero modules). An essential difference between Roberts' definition of almost Cohen-Macaulay algebras and ours is that it is not obvious from our version at all, whether $B/\fm B$ is almost zero, or not (which is why the adverb ``weakly'' appears above). We also remark that Roberts' version leads to a proof of the Monomial Conjecture in mixed characteristic, while our version does not. However, our idea still seems to have some potential force. Here is our main theorem:

\begin{main theorem}
Let $(R,\fm)$ be a complete local domain of mixed characteristic $p > 0$. Then there exist a system of parameters $p,x_{2},\ldots,x_{d}$ of $R$ and a weakly almost Cohen-Macaulay $R^+$-algebra $B$ satisfying the following conditions: 

\begin{enumerate}
\item[$\mathrm{(1)}$]
$(p,x_{2},\ldots,x_{d})B \ne B$;

\item[$\mathrm{(2)}$]
$x_{2},\ldots,x_{d}$ forms a regular sequence on $B/pB$; 

\item[$\mathrm{(3)}$]
$p$ is not nilpotent in $B$ and the ideal $(0:_{B} p)$ is annihilated by $p^{\epsilon}$ for any rational $\epsilon > 0$.

\end{enumerate}
\end{main theorem}

We note that for a valuation $v$ on $R^{+}$, we have $v(p^{\epsilon})=\epsilon \cdot v(p)$. To say that $B$ is a big Cohen-Macaulay algebra ``over $R^+$'' causes no confusion at all, since $R^+$ is the directed union of its module-finite subextensions over $R$. Therefore, $B$ is weakly almost Cohen-Macaulay over any module-finite extension domain of $R$. The difficult part of the above theorem is to show that the element $p$ is ``almost'' regular on $B$. The proof is to reduce to the positive characteristic case via Fontaine rings and we construct a certain perfect algebra over it. Then we lift it to the ring of Witt vectors. We will give a brief review on the theory of Fontaine rings and Witt vectors for the absolute integral closures of complete local domains.

\section{Preliminaries}

In this article, $(R,\fm)$ will denote a local Noetherian ring. Let us say that a domain $A$ has \textit{mixed characteristic} $p>0$ if $A$ has characteristic zero, while $A/pA$ has characteristic $p>0$. Let $x_{1},\ldots,x_{n}$ be a sequence in a ring $A$ and let $N$ be an $A$-module. Then the sequence $x_{1},\ldots,x_{n}$ is said to be \textit{N-regular} if $(x_{1},\ldots,x_{n})N \ne N$ and $x_{k}$ is a nonzero divisor of $N/(x_{1},\ldots,x_{k-1})N$ for all $1 \le k \le n$.

The \textit{absolute integral closure} of an integral domain $R$ is defined to be the integral closure of $R$ in an algebraic closure of its field of fractions and denote it by $R^{+}$. The symbol employed in~\cite{Ar} for $R^{+}$ is different from ours. We will use the following fact later. Let $R$ be any domain and let $P$ be any prime ideal of $R^{+}$. Then we have $R^{+}/P \simeq (R/R \cap P)^{+}$. Let $A$ be any ring of positive characteristic and let $C:=A_{\mathrm{red}}$. Then the \textit{perfect closure} of $A$ is defined as the direct limit of the system on the top or bottom defined by the Frobenius map:
$$
\begin{CD}
C @>>> C^{1/p} @>>> C^{1/p^2} @>>> \cdots \\
@VVV @V \wr V\mathbf{F}V @V \wr V\mathbf{F}^{2}V \\
C @>\mathbf{F}_{C}>> C @>\mathbf{F}_{C}>> C @>\mathbf{F}_{C}>> \cdots \\
\end{CD}
$$
in which the first vertical arrow is the identity map and the horizontal arrows on the top are the natural inclusions. Since the iterates of the Frobenius map annihilates the nilradical of $A$, the perfect closure may also be defined as $\varinjlim_{\mathbf{F}_{A}} A$. Furthermore, let $I \subseteq A$ be an ideal. Then let $I^{[p^e]}$ denote the ideal of $A$ generated by $x^{p^e}$ for all $x \in I$. 

The following theorem states that there are canonical big Cohen-Macaulay algebras in positive characteristic. The advantage of working with the absolute integral closures is that it allows us to extend it to the weakly functorial case, and they are not too large to deal with.

\begin{Theorem}[Hochster,Huneke~\cite{HH1}; Huneke,Lyubeznik~\cite{HL}]
\label{HHHL}
Let $(R,\fm)$ be a local domain of characteristic $p>0$. Assume one of the following conditions:
\begin{enumerate}
\item[$\mathrm{(1)}$]
$R$ is an excellent local domain.

\item[$\mathrm{(2)}$]
$R$ is a homomorphic image of a Gorenstein local ring.
\end{enumerate}
Then every system of parameters for $R$ is a regular sequence on $R^{+}$.
\end{Theorem}

We need the following simple fact for later use.

\begin{Lemma}
\label{completion}
Let $x$ be a regular element in a ring $R$, and let $J$ be an ideal of $R$ such that $R \ne xR+J$. Suppose that $R$ is $x$-adically complete and $R/xR$ is $J$-adically complete. Then $R$ is complete in the $(xR+J)$-adic topology.
\end{Lemma}

\begin{proof}
We first show that $R/x^{n}R$ is $J$-adically complete. If $n=1$, this is so by assumption. For any $k<n$, we assume that $R/x^{k}R$ is $J$-adically complete. Since $x$ is regular, we have
$$
x^{n-1}:R/xR \simeq x^{n-1}R/x^{n}R,
$$
which implies that $x^{n-1}R/x^{n}R$ is $J$-adically complete. Applying the five-lemma to the short exact sequence
$$
\begin{CD}
0 @>>> x^{n-1}R/x^{n}R @>>> R/x^{n}R @>>> R/x^{n-1}R @>>> 0, \\
\end{CD}
$$
we deduce that $R/x^{n}R$ is $J$-adically complete. Finally, we have 
$$
\varprojlim_{k \in \mathbb{N}} R/(xR+J)^{k} \simeq \varprojlim_{m,n \in \mathbb{N}} R/(x^{m}R+J^{n}) \simeq \varprojlim_{m \in \mathbb{N}}\big(\varprojlim_{n \in \mathbb{N}} R/(x^{m}R+J^{n})\big) \simeq \varprojlim_{m \in \mathbb{N}} R/x^{m}R \simeq R,
$$
which is the required claim.
\end{proof}

\section{Algebra modifications and some criteria}

Let $T$ be an algebra over a local ring $(R,\fm)$. We begin with the notion of algebra modifications due to Hochster.

\begin{Definition}[algebra modifications]
Let $x_{1},\ldots,x_{k+1}$ be a sequence in a local ring $(R,\fm)$, and let $t_{1},\ldots,t_{k+1}$ be a sequence in $T$ such that $x_{k+1} t_{k+1}=\sum_{i=1}^{k} x_{i} t_{i}$. Let $X_{1},\ldots,X_{k}$ be a set of indeterminates over $T$. Then we say that 
$$
T'=\frac{T[X_{1},\ldots,X_{k}]}{(t_{k+1}-\sum_{i=1}^{k} x_{i} X_{i})}
$$ 
is an \textit{algebra modification of T}. We define a \textit{sequence of algebra modifications}:
$$
\begin{CD}
T=T_{0} @>>> T_{1} @>>> \cdots @>>> T_{s} @>>> \cdots \\
\end{CD}
$$
such that every $T_{i+1}$ is an algebra modification of $T_{i}$.
\end{Definition}

We will consider the algebra modifications in the case where the sequence $x_{1},\ldots,x_{d}$ is a system of parameters of $R$. The point is that if there is a relation on a system of parameters, one extends the ring to trivialize it. In order to make the notion effective, we shall need to introduce a more universal object (see~\cite{HH2} for more details).

Let $d=\dim R$ for a local ring $(R,\fm)$ and let $\mathcal{F}$ denote a fixed non-empty family of sequences of length $d$, all of which form systems of parameters for $R$.  Let $\mathcal{S}_{T}$ denote the set of all relations of every element of $\mathcal{F}$ for an $R$-algebra $T$, whose precise meaning is as follows: Let $x_{1},\ldots,x_{k+1}$ be an initial segment of an element of $\mathcal{F}$. Then the relation $\lambda: x_{k+1}t_{k+1}-\sum_{i=1}^{k} x_{i}t_{i}=0$ with $t_{i} \in T$ is an object of $\mathcal{S}_{T}$. We refer $t(\lambda)=k$ as a \textit{type} of the relation $\lambda$. Note that $0 \le t(\lambda)<d$. In particular, if $t(\lambda)=0$, then the relation is merely $x t=0$. See~\cite{Ho94} for the following definition.

\begin{Definition}
Let $T$ be an algebra over a local ring $(R,\fm)$. Let $\mathcal{J}(T/R)$ be an ideal of the polynomial algebra $T[X_{\lambda,j}]=T[X_{\lambda,j};\lambda \in \mathcal{S}_{T},1 \le j \le t(\lambda)]$, generated by all the polynomials $(t_{k+1}-\sum_{i=1}^{k} x_{i} X_{\lambda,i})$ associated to the relation $x_{k+1}t_{k+1}=\sum_{i=1}^{k} x_{i}t_{i}$ with $t_{i} \in T$. We set
$$
\Mod(T/R)=\frac{T[X_{\lambda,j}]}{\mathcal{J}(T/R)}.
$$
We define $\Mod_{n}(T/R)$ recursively: $\Mod_{0}(T/R)=T$ and $\Mod_{n+1}(T/R)=\Mod(\Mod_{n}(T/R)/R)$. Finally we let
$$
\Mod_{\infty}(T/R)=\varinjlim_{n \in \mathbb{N}}\Mod_{n}(T/R).
$$
\end{Definition}

We shall denote by $\mathcal{T}$ a single sequence of algebra modifications:
$$
\begin{CD}
T=T_{0} @>>> T_{1} @>>> \cdots @>>> T_{i} @>>> \cdots \\
\end{CD}
$$
over the local ring $(R,\fm)$. We shall say that $\mathcal{T}$ is \textit{bad} if $1 \in \fm T_{s}$ for some $s \ge 0$.

\begin{Remark}
\label{remark}
\begin{enumerate}
\item
The most important case for us is when $\mathcal{F}$ consists of a single system of parameters $x_{1},\ldots,x_{d}$, in which case one keeps track of the relations with respect to an initial segment of $x_{1},\ldots,x_{d}$. The object $\Mod_{\infty}(T/R)$ defined above is a possibly improper big Cohen-Macaulay $R$-algebra with respect to $\mathcal{F}$ meaning that every element of $\mathcal{F}$ is a regular sequence on $\Mod_{\infty}(T/R)$. So the difficulty with $\Mod_{\infty}(T/R)$ is in showing that whether $1 \in (x_{1},\ldots,x_{d})\Mod_{\infty}(T/R)$, or not.

\item
If an $R$-algebra $T$ maps to a possibly improper big Cohen-Macaulay $R$-algebra $W$, then one can map any sequence of algebra modifications of $T$ to $W$. By this fact, we see that $\Mod_{\infty}(T/R) \ne 0$.
\end{enumerate}
\end{Remark}

\begin{Definition}
Let $T$ be an algebra over a local ring $(R,\fm)$ and let $0 \le k < \dim(R)$ be fixed. Let $\mathcal{F}$ be a non-empty fixed family of sequences of $R$. Then we say that a sequence of algebra modifications with respect to $\mathcal{F}$ is \textit{of type $\ge$ k} if every $T_{i+1}$ is a modification of $T_{i}$ with respect to a relation of type at least $k$.
\end{Definition}

In analogy with the construction of $\Mod(T/R)$, we define $\Mod(T/R)_{\ge k}$ as a polynomial algebra over $T$ modulo an ideal, whose generators come from all the relations of type $\ge k$ with respect to sequences in the family $\mathcal{F}$. We may also define $\Mod_{i}(T/R)_{\ge k}$ and $\Mod_{\infty}(T/R)_{\ge k}$ as well. Suppose that
$$
\begin{CD}
T=T_{0} @>>> T_{1} @>>> \cdots @>>> T_{s}
\end{CD}
$$
is a finite sequence of modifications of type $\ge k$ with respect to $\mathcal{F}$. Then we can inductively construct the commutative diagram of $T$-algebra homomorphisms:
$$
\begin{CD}
\Mod_{0}(T/R)_{\ge k} @>>> \Mod_{1}(T/R)_{\ge k} @>>> \cdots @>>> \Mod_{s}(T/R)_{\ge k} \\
@AAA @AAA @. @AAA \\
T_{0} @>>> T_{1} @>>> \cdots @>>> T_{s} \\
\end{CD}
$$
Thus, we have the following proposition.

\begin{Proposition}
\label{type}
Let $T$ be an algebra over a local ring $(R,\fm)$ with a system of parameters $x_{1},\ldots,x_{d}$, and let $0 \le k < \dim R$ be fixed. Then the following conditions are equivalent:

\begin{enumerate}
\item[$\mathrm{(1)}$]
There exists a $T$-algebra $B$ such that $1 \notin (x_{1},\ldots,x_{d})B$ and $x_{k+1},\ldots,x_{d}$ forms a regular sequence on $B/(x_{1},\ldots,x_{k})B$.

\item[$\mathrm{(2)}$]
Suppose that
$$
\begin{CD} 
T=T_{0} @>>> T_{1} @>>> \cdots @>>> T_{s}
\end{CD}
$$
is any finite sequence of modifications of $T$ of type $\ge k$ with respect to $x_{1},\ldots,x_{d}$. Then we have $1 \notin (x_{1},\ldots,x_{d})T_{s}$.
\end{enumerate}

\end{Proposition}

\begin{proof}
$(1) \Rightarrow (2)$: This will be done as follows. By assumption, the $T$-algebra $\Mod_{\infty}(T/R)_{\ge k} $ maps to $B$ and there is a $T$-algebra homomorphism from a sequence of modifications $T=T_{0} \to T_{1} \to \cdots \to T_{s}$ to $\Mod_{\infty}(T/R)_{\ge k}$. Hence $1 \notin (x_{1},\ldots,x_{d})T_{s}$.
 
$(2) \Rightarrow (1)$: By the construction of $\Mod_{\infty}(T/R)_{\ge k}$, it suffices  to show that 
$$
1 \notin (x_{1},\ldots,x_{d})\Mod_{\infty}(T/R)_{\ge k}.
$$ 
Suppose the contrary. Then we have $1=\sum_{i=1}^{d}x_{i}t_{i}$ for $t_{i} \in \Mod_{s}(T/R)_{\ge k}$ for some $s \ge 0$. Since we need only finitely many relations used in the construction of $\Mod_{s}(T/R)_{\ge k}$ for the presentation of $1$, we can construct a map: $\Mod_{s-1}(T/R)_{\ge k} \to T_{s}'$ which is just a finite sequence of modifications of $\Mod_{s-1}(T/R)_{\ge k}$ so as to have $1 \in (x_{1},\ldots,x_{d})T_{s}'$. 

All of these modifications can be described using finitely many elements of $\Mod_{s-1}(T/R)_{\ge k}$. All the elements and relations needed will be in a modification of $\Mod_{s-2}(T/R)_{\ge k}$ with respect to only finitely many relations.
Therefore, we may keep track of $\Mod_{i}(T/R)_{\ge k}$ backward until we have arrived at $T=\Mod_{0}(T/R)_{\ge k}$. Hence $T_{s}'$ is obtained from $T$ after finitely many steps of modifications satisfying $1 \in (x_{1},\ldots,x_{d})T_{s}'$, which is a contradiction.
\end{proof}

\section{Fontaine rings of the absolute integral closures and Witt vectors}

In this section, we discuss some structure of Fontaine rings that is not found in ~\cite{GaRa}, since the idea of Fontaine rings is not prevalent in standard commutative algebra. In particular, we show that the Fontaine ring of the absolute integral closure of a complete local domain contains a complete regular local ring, which is constructed as a projective limit defined by the Frobenius map. We then lift the Fontaine ring to the ring of Witt vectors. For the Witt vectors, we refer the reader to Serre's book~\cite{Se}. For a complete theory of Fontaine rings with its relation to Witt vectors that we use, we refer to Gabber and Ramero~\cite{GaRa}. 

Now assume that $(R,\fm)$ is a complete regular local ring of mixed characteristic $p>0$ with perfect residue field and assume that $p,x_{2},\ldots,x_{d}$ is a regular system of parameters of $R$. We denote by $R^{+}$ its absolute integral closure.

\begin{Definition}[Fontaine ring]
Let the notation be as above. Then we define
$$
\mathbf{E}(R^{+})=\varprojlim_{n \in \mathbb{N}} A_{n} 
$$
in which $A_{n}=R^{+}/pR^{+}$ for every $n \in \mathbb{N}$ and $A_{n+1} \to A_{n}$ is the Frobenius map.
\end{Definition}

Any nonzero element of $\mathbf{E}(R^{+})$ is of the form $\langle x \rangle=(x,x^{\frac{1}{p}},x^{\frac{1}{p^2}},\ldots)$. This notation is slightly abused, since there are ambiguities in choosing $p$-power elements. However, as we shall discuss the Fontaine rings for the absolute integral closures exclusively and the special choices of $p$-power roots are not important for us, this will not cause any kind of issues. We define a natural surjective ring homomorphism: 
$$
\Phi_{R^{+}}^{n}:\mathbf{E}(R^{+}) \to R^{+}/pR^{+}
$$
by the rule $(a_{0},a_{1},a_{2},\ldots) \mapsto a_{n}$. We write $\Phi_{R^{+}}:=\Phi_{R^{+}}^{0}$ for simplicity.

\begin{Remark}
\begin{enumerate}
\item
It is easy to see from the definition that $\mathbf{E}(R^{+})$ is a perfect ring of characteristic $p>0$. If $\dim R=1$, then it is known that $\mathbf{E}(R^{+})$ is a valuation ring. In particular, it is a domain (\cite{GaRa}, Lemma 5.2.27).

\item
There is an alternate way of defining Fontaine rings (see~\cite{FonOuy} for the detail). Let $\widehat{R^{+}}$ be the $p$-adic completion of $R^{+}$. Then it is defined as:
$$
\mathbf{E}(R^{+}) \simeq \big((x_{0},x_{1},x_{2},\ldots)~|~x_{i} \in \widehat{R^{+}}, x_{i+1}^{p}=x_{i}\big),
$$
in which the multiplicative structure is given by the one on $\widehat{R^{+}}$ and the additive structure is given by the rule:
$$
(\ldots,x_{m},\ldots)+(\ldots,y_{m},\ldots)=\big(\ldots,\lim_{n \to \infty} (x_{m+n}+y_{m+n})^{p^{n}},\ldots \big).
$$
\end{enumerate}
\end{Remark}

\begin{Lemma}
\label{normal}
Suppose that $A$ is a normal domain, $k > 0$ is an integer, and $p$ is a prime integer. Then we have the following statements:

\begin{enumerate}
\item[$\mathrm{(1)}$]
Suppose that $A$ has mixed characteristic $p>0$ and $x^{p^k}-p=0$ has a root in $A$. Then the $k$-th iterated Frobenius map on $A/pA$ induces an injection: $A/p^{1/p^{k}}A \to A/pA$.

\item[$\mathrm{(2)}$]
Let $A \to B$ be an integral extension of integral domains and let $t \in A$ be any nonzero element. Then the induced map: $A/tA \to B/tB$ is injective.
\end{enumerate}
\end{Lemma}

\begin{proof}
$(1)$: We denote by $\mathbf{F}_A^{k}:A/pA \to A/pA$ the $k$-th iterated Frobenius map. Assume that $\mathbf{F}_A^{k}(\overline{t})=0$ for $t \in A$. Then we have $t^{p^k}=p \cdot \theta$ for some $\theta \in A$ and $t=p^{1/p^k} \cdot \tilde{\theta}$ in $K$, where $K$ is the algebraic closure of the field of fractions of $A$ and $\tilde{\theta}$ is a root of the equation $x^{p^k}-\theta=0$. But since $\tilde{\theta}=t \cdot p^{-1/p^k} \in A[p^{-1}]$, we get $\tilde{\theta} \in A$ by the normality of $A$. Hence $t \in p^{1/p^k}A$.

$(2)$: This is an easy exercise using the normality of $A$.
\end{proof}

This immediately implies the following. Let $B$ be any domain of mixed characteristic $p>0$. Then the Frobenius map $B^{+}/pB^{+} \to B^{+}/pB^{+}$ induces a ring isomorphism: 
$$
B^{+}/p^{1/p}B^{+} \simeq B^{+}/pB^{+}.
$$

Recall that $R$ is a complete regular local ring with a regular system of parameters $p,x_{2},\ldots,x_{d}$. Then we may define a sequence of module-finite extensions of regular local rings:
$$
\begin{CD}
R=R_{0} @>>> R_{1} @>>> R_{2} @>>> \cdots \\
\end{CD}
$$
such that $R_{n}:=R[p^{1/p^n},x_{2}^{1/p^n},\ldots,x_{d}^{1/p^n}]$, and the Frobenius map surjects $R_{n+1}/pR_{n+1}$ onto $R_{n}/pR_{n}$, which is naturally the subring of $R_{n+1}/pR_{n+1}$.

\begin{Definition}[small Fontaine ring]
Let $(R,\fm)$ be as above. Then we define
$$
\mathbf{E}(R)^{\times}=\varprojlim_{n \in \mathbb{N}} A_{n} 
$$
such that $A_{n}=R_{n}/pR_{n}$ and $A_{n+1} \to A_{n}$ is the Frobenius map.
\end{Definition}

Note that there is the following sequence of isomorphisms of local rings defined by the Frobenius map:
$$
R_{k+1}/p^{1/p^k}R_{k+1} \simeq R_{k}/p^{1/p^k}R_{k} \simeq \cdots \simeq  R/pR,
$$
and thus the sequence $\langle p \rangle,\langle x_{2} \rangle,\ldots, \langle x_{d} \rangle$ is contained in $\mathbf{E}(R)^{\times}$. For a sequence $(R_{k}~|~k \in \mathbb{N})$ as above, we define a natural ring homomorphism:
$$
\Phi_{R}^{k}:\mathbf{E}(R)^{\times} \to R_{k}/pR_{k}
$$
by the rule $(a_{0},a_{1},a_{2},\ldots) \mapsto a_{k}$, and this map is surjective for all $n \ge 0$. Some important properties of the (small) Fontaine rings are contained in the following proposition.

\begin{Proposition}
\label{Fontaine}
Let the notation be as above.

\begin{enumerate}
\item[$\mathrm{(1)}$]
$\mathbf{E}(R^{+})$ is a $\langle p \rangle$-adically complete quasilocal algebra and fits into the following short exact sequence:
$$
\begin{CD}
0 @>>> \mathbf{E}(R^{+}) @>\langle p \rangle>> \mathbf{E}(R^{+}) @>>> R^{+}/pR^{+} @>>> 0. \\
\end{CD}
$$
In particular, $\langle p \rangle$ is a nonzero divisor of $\mathbf{E}(R^{+})$.

\item[$\mathrm{(2)}$]
$\mathbf{E}(R)^{\times}$ is a complete regular local ring such that $\langle p \rangle,\langle x_{2} \rangle,\ldots,\langle x_{d} \rangle$ is a regular system of parameters and the residue field of $\mathbf{E}(R)^{\times}$ is the same as that of $R$. Finally, there is a short exact sequence:
$$
\begin{CD}
0 @>>> \mathbf{E}(R)^{\times} @>\langle p \rangle>> \mathbf{E}(R)^{\times} @>>> R/pR @>>> 0. \\
\end{CD}
$$
\end{enumerate}
\end{Proposition}

\begin{proof}
$(1)$: This is the content of (\cite{GaRa}, Proposition 5.2.33).

$(2)$: First, we show that the sequence 
$$
\begin{CD}
0 @>>> \mathbf{E}(R)^{\times} @>\langle p \rangle >> \mathbf{E}(R)^{\times} @>\Phi_{R}>> R/pR @>>> 0
\end{CD}
$$
is short exact. It only suffices to show that the middle part in the sequence is exact, since $\mathbf{E}(R)^{\times}$ is a subring of $\mathbf{E}(R^{+})$. This follows easily by looking into the following sequence of isomorphisms: 
$$
R_{k+1}/p^{1/p^{k+1}}R_{k+1} \simeq R_{k}/p^{1/p^{k}}R_{k} \simeq \cdots \simeq R/pR.
$$ 
Next, as discussed in~\cite{GaRa}, there is the following commutative diagram: 
$$ 
\begin{CD}
\mathbf{E}(R)^{\times} @>\langle p \rangle^{p^{k+1}}>> \mathbf{E}(R)^{\times} @>\Phi^{k+1}_{R}>> R_{k+1}/pR_{k+1} @>>> 0 \\
@V \langle p \rangle^{p^{k}(p-1)} VV @| @V\mathbf{F}VV \\
\mathbf{E}(R)^{\times} @>\langle p \rangle^{p^k}>> \mathbf{E}(R)^{\times} @>\Phi^{k}_{R}>> R_{k}/pR_{k} @>>> 0
\end{CD}
$$
where every row is an exact sequence and $\mathbf{F}$ is the Frobenius map. Thus we deduce the following isomorphism:
$$
\varprojlim_{k \in \mathbb{N}} \mathbf{E}(R)^{\times}/\langle p \rangle^{p^k}\mathbf{E}(R)^{\times} \simeq \mathbf{E}(R)^{\times}, 
$$
showing that $\mathbf{E}(R)^{\times}$ is $\langle p \rangle$-adically complete and separated. Moreover, we have
$$
\mathbf{E}(R)^{\times}/\langle p \rangle\mathbf{E}(R)^{\times} \simeq R/pR \simeq k[[x_{2},\ldots,x_{d}]]
$$
from the above short exact sequence. Therefore, we find that $(\langle p \rangle,\langle x_{2} \rangle,\ldots,\langle x_{d} \rangle)$ is a maximal ideal of $\mathbf{E}(R)^{\times}$ and $\dim \mathbf{E}(R)^{\times} \ge d$. To complete the proof, it suffices to show that $\mathbf{E}(R)^{\times}$ is a complete local Noetherian ring. Lemma~\ref{completion} shows that $\mathbf{E}(R)^{\times}$ is complete and separated in the $(\langle p \rangle,\langle x_{2} \rangle,\ldots,\langle x_{d} \rangle)$-adic topology, which also shows that $(\langle p \rangle,\langle x_{2} \rangle,\ldots,\langle x_{d} \rangle)$ is the unique maximal ideal. Now (\cite{Mat}, Theorem 29.4) shows that $R$ is Noetherian, as claimed.
\end{proof}

Henceforth, we will view $\mathbf{E}(R)^{\times} \to \mathbf{E}(R^{+})$ as a structure homomorphism of the Fontaine ring of $R^{+}$, which reduces modulo $\langle p \rangle$ to an integral extension $R/pR \to R^{+}/pR^{+}$. By some calculated examples of Roberts~\cite{Ro2}, the Krull dimension of a ring $A$ with $\mathbf{E}(R)^{\times} \subseteq A \subseteq \mathbf{E}(R^{+})$ may be quite large.

Towards the construction of almost Cohen-Macaulay algebras in mixed characteristic, it is necessary to use the theory of the ring of Witt vectors to lift rings of positive characteristic to rings of mixed characteristic. Let $W(\mathbf{E}(R^{+}))$ denote the ring of Witt vectors. Now since $\mathbf{E}(R^{+})$ is a perfect algebra, it follows that $W(\mathbf{E}(R^{+}))$ is a $p$-adically complete and separated algebra, there is an isomorphism: $W(\mathbf{E}(R^{+}))/p \cdot W(\mathbf{E}(R^{+})) \simeq \mathbf{E}(R^{+})$, and $p$ is a nonzero divisor of $W(\mathbf{E}(R^{+}))$. Moreover, Lemma~\ref{completion} shows that $W(\mathbf{E}(R^{+}))$ is complete and separated in the $(p,\langle p \rangle)$-adic topology. Let $\widehat{R^{+}}$ be the $p$-adic completion of $R^{+}$. For the proof of the following lemma, see (\cite{GaRa}, Proposition 5.2.21).

\begin{Lemma}
\label{Witt}
Keeping the notation as above, there is a commutative diagram:
$$
\begin{CD}
\mathbf{E}(R^{+}) @>\widehat{\Phi}_{R^{+}}>> \widehat{R^{+}} \\
@| @V \pi VV \\
\mathbf{E}(R^{+}) @>\Phi_{R^{+}}>> R^{+}/pR^{+}
\end{CD}
$$
in which $\pi$ is the natural projection and $\widehat{\Phi}_{R^{+}}$ is uniquely determined such that $\widehat{\Phi}_{R^{+}}(\langle 1 \rangle)=1$ and $\widehat{\Phi}_{R^{+}}$ is a multiplicative map. In fact, we have $\widehat{\Phi}_{R^{+}}(\langle x\rangle)=x$ for any lift $\langle x \rangle \in \mathbf{E}(R^{+})$ of every element $x \in R^{+}$.
\end{Lemma}

Let $\theta_{\mathbf{E}(R^{+})}:\mathbf{E}(R^{+}) \to W(\mathbf{E}(R^{+}))$ be defined by $\theta_{\mathbf{E}(R^{+})}(a)=(a,0,\ldots,0,\ldots)$. This map is multiplicative and is called the \textit{Teichm$\ddot{u}$ller lift}. However, $\theta_{\mathbf{E}(R^{+})}$ is not additive. 

Let now $\underline{a}=(a_{0},\ldots,a_{n},\ldots) \in W(\mathbf{E}(R^{+}))$. Since $\mathbf{E}(R^{+})$ is a perfect ring, for every $a \in \mathbf{E}(R^{+})$, we find a unique element $x \in \mathbf{E}(R^{+})$ such that $x^{p^n}=a$. Denote this element by $a^{p^{-n}}$. Then we define the map:
$$
\psi:W(\mathbf{E}(R^{+})) \to \widehat{R^{+}}
$$
by the rule 
$$
\psi(\underline{a})=\sum_{n=0}^{\infty}p^{n} \cdot \widehat{\Phi}_{R^{+}}(a_{n}^{p^{-n}}),
$$
in which the right hand side makes sense in $\widehat{R^{+}}$. It is easy to check that $\psi \circ \theta_{\mathbf{E}(R^{+})}=\widehat{\Phi}_{R^{+}}$, where $\theta_{\mathbf{E}(R^{+})}$ is as above. The next proposition claims that $\psi$ defines a surjective ring homomorphism (see~\cite{GaRa}, Proposition 5.2.33).

\begin{Proposition}
\label{Teichmuller}
Under the notation as above, the map $\psi:W(\mathbf{E}(R^{+})) \to \widehat{R^{+}}$ is a ring homomorphism that fits into a short exact sequence:
$$
\begin{CD}
0 @>>> W(\mathbf{E}(R^{+})) @>\vartheta>> W(\mathbf{E}(R^{+})) @>>> \widehat{R^{+}} @>>> 0 \\
\end{CD}
$$
for $\vartheta:=\theta_{\mathbf{E}(R^{+})}(\langle p \rangle)-p$. Moreover, $\vartheta, p$ forms a regular sequence on $W(\mathbf{E}(R^{+}))$. 
\end{Proposition}

Here is a quite important remark regarding the above proposition. Since the ring $R^{+}$ maps to its $p$-adic completion $\widehat{R^+}$, there is a natural ring homomorphism:
$$
R^{+} \to \widehat{R^+} \simeq \frac{W(\mathbf{E}(R^{+}))}{\vartheta \cdot W(\mathbf{E}(R^{+}))},
$$
which is not surjective. In fact, this fact may be regarded as a natural generalization of the construction of complete discrete valuation rings in mixed characteristic as the ring of Witt vectors of perfect fields.

\section{Statement and proof of the main theorems}

In order to prove the main theorem, we need to construct a certain perfect algebra over the Fontaine ring. Then its ring of Witt vectors will be a $p$-adically complete and separated algebra, in which the natural lift of the sequence under consideration that comes from the Fontaine ring forms a weakly almost regular sequence. The main technique using sequences of modifications to produce the desired algebra in the theorem is due to Hochster.

\begin{Discussion}
Let $(R,\fm)$ be a complete regular local ring with a regular system of parameters $p,x_2,\ldots,x_d$ such that the residue field is perfect. Then we get from Proposition~\ref{Fontaine} that $\mathbf{E}(R^{+})/\langle p \rangle \mathbf{E}(R^{+}) \simeq R^{+}/pR^{+}$, so that $\langle p \rangle, \langle x_{i} \rangle$ is a regular sequence on $\mathbf{E}(R^{+})$. We can show that every $\langle x_{i} \rangle$ is a nonzero divisor of $\mathbf{E}(R^{+})$ as follows. Suppose that we have $\langle x_{i} \rangle m=0$ for $m \in \mathbf{E}(R^{+})$. Since $\mathbf{E}(R^{+})$ is $\langle p \rangle$-adically separated, it suffices to show that $m \in \langle p \rangle^{t} \mathbf{E}(R^{+})$ for all $t>0$. Since $\langle x_{i} \rangle$ is regular on $\mathbf{E}(R^{+})/\langle p \rangle \mathbf{E}(R^{+})$, we have $m \in \langle p \rangle \mathbf{E}(R^{+})$. Suppose that $m=\langle p \rangle^{t-1}m'$ for some $m'$. Then we have $0=\langle x_{i} \rangle m=\langle x_{i} \rangle \langle p \rangle^{t-1}m'$, which implies that $\langle x_{i} \rangle m'=0$. Hence we have $m'=\langle p \rangle m''$ for some $m''$ and $m=\langle p \rangle^{t}m''$, which is the claim. 

Let $\mathbf{E}(R^{+})_{\langle x_{2} \rangle \cdots \langle x_{d} \rangle}$ be the localization of $\mathbf{E}(R^{+})$ with respect to the element $\langle x_{2} \rangle \cdots \langle x_{d} \rangle$. Then we find that $\langle p \rangle$ is a nonzero divisor of $\mathbf{E}(R^{+})_{\langle x_{2} \rangle \cdots \langle x_{d} \rangle}$, as follows from the short exact sequence in Proposition~\ref{Fontaine}. Pick any relation
$$
\langle x_{k+1} \rangle a_{k+1}=\langle p \rangle a_{1}+\sum_{i=2}^{k} \langle x_{i} \rangle a_{i}
$$ 
for $k > 0$ and $a_{i} \in \mathbf{E}(R^{+})_{\langle x_{2} \rangle \cdots \langle x_{d} \rangle}$. Then since $\langle x_{k+1}\rangle$ is a unit in $\mathbf{E}(R^{+})_{\langle x_{2} \rangle \cdots \langle x_{d} \rangle}$, the above relation forces $a_{k+1} \in (\langle p \rangle,\ldots,\langle x_{k} \rangle)$, that is, $\langle p \rangle,\ldots,\langle x_{d} \rangle$ forms an improper regular sequence on  $\mathbf{E}(R^{+})_{\langle x_{2} \rangle \cdots \langle x_{d} \rangle}$. Thus, any sequence of  algebra modifications of $\mathbf{E}(R^{+})$ can be mapped to  $\mathbf{E}(R^{+})_{\langle x_{2} \rangle \cdots \langle x_{d} \rangle}$. This fact will play an important role later.
\end{Discussion}

Henceforth, we continue to use the notation $\mathbf{E}(R^{+})_{\langle x_{2} \rangle \cdots \langle x_{d} \rangle}$. Then by (\cite{Die}, Lemma 3.5), $\mathbf{E}(R^{+})_{\langle x_{2} \rangle \cdots \langle x_{d} \rangle}$ is a perfect algebra. Before stating the main theorems, we make a remark regarding the proof of the first main theorem. In view of Remark \ref{remark}, one might wonder if one could take $\mathbf{E}(R^{+})_{\langle p \rangle \cdots \langle x_{d} \rangle}$ instead of $\mathbf{E}(R^{+})_{\langle x_{2} \rangle \cdots \langle x_{d} \rangle}$. The reason for doing so is that we need that $\theta_{\mathbf{E}(R^{+})}(\langle p \rangle)-p$ is not a unit in the Witt ring $W(\mathbf{E}(R^{+})_{\langle x_{2} \rangle \cdots \langle x_{d} \rangle})$. In fact, if $T$ is a $p$-adically complete and separated ring, any element of the form $p+u \in T$ is a unit, if $u$ is a unit. Based upon the above discussions, we are ready to prove the following theorem.

\begin{Theorem}
\label{Main1}
Under the notation as above, there exists an $\mathbf{E}(R^{+})$-algebra $S$ satisfying the following conditions:
 
\begin{enumerate}
\item[$\mathrm{(1)}$]
$(\langle p \rangle,\langle x_{2} \rangle,\ldots,\langle x_{d}\rangle)S \ne S$; 

\item[$\mathrm{(2)}$]
$\langle x_{2} \rangle,\ldots,\langle x_{d} \rangle$ forms a regular sequence on $S/\langle p \rangle S$; 

\item[$\mathrm{(3)}$]
$\langle p \rangle$ is not nilpotent and the ideal $(0:_{S} \langle p \rangle)$ is annihilated by $\langle p \rangle^{\epsilon}$ for any rational $\epsilon > 0$;

\item[$\mathrm{(4)}$]
$S$ is a perfect algebra.
\end{enumerate}
Moreover, there is an $\mathbf{E}(R^{+})$-algebra homomorphism $S \to \mathbf{E}(R^{+})_{\langle x_{2} \rangle \cdots \langle x_{d} \rangle}$.
\end{Theorem}

\begin{proof}
Before we start the proof, we note that $\langle p \rangle^{\epsilon} \in \mathbf{E}(R^{+})$ for any rational $\epsilon > 0$. We will construct the $\mathbf{E}(R^{+})$-algebra $S$ by taking sequences of modifications of $\mathbf{E}(R^{+})$, using the relations of type $\ge 1$ with respect to $z_{1}:=\langle p \rangle, z_{2}:=\langle x_{2} \rangle,\ldots,z_{d}:=\langle x_{d} \rangle$. We prove the theorem by contradiction. Suppose that, as in Proposition \ref{type}, there is a sequence of modifications of type $\ge 1$ with respect to $z_{1},\ldots,z_{d}$:
$$
\begin{CD}
\mathcal{T}: \mathbf{E}(R^{+})=T_{0} @>>> T_{1} @>>> \cdots @>>> T_{s} \\
\end{CD}
$$
such that $1 \in (z_{1},\ldots,z_{d})T_{s}$. Note that $\mathbf{E}(R^{+})/z_{1}\mathbf{E}(R^{+}) \simeq R^{+}/pR^{+}$ is an algebra over a complete regular local ring $\mathbf{E}(R)^{\times}/z_{1} \mathbf{E}(R)^{\times} \simeq R/pR$ on which $z_{2},\ldots,z_{d}$ descends to a system of parameters. We keep the same notation for a system of parameters of $R/pR$.

After dividing the sequence $\mathcal{T}$ out by $z_{1}$, we show that the induced sequence maps to a sequence of modifications of $R^{+}/pR^{+}$. Now let us look at things more closely. Let
$$
T_{i+1}=\frac{T_{i}[X_{1}^{(i)},\ldots,X_{k}^{(i)}]}{(s_{k+1}^{(i)}-\sum_{j=1}^{k}z_{j}X_{j}^{(i)})},~s_{k+1}^{(i)} \in T_{i}.
$$
Then we have 
$$
T_{i+1} \equiv \frac{T_{i}[X_{1}^{(i)},\ldots,X_{k}^{(i)}]}{(s_{k+1}^{(i)}-\sum_{j=2}^{k}z_{j}X_{j}^{(i)})} \mod z_{1}.
$$
We set
$$
\overline{T}_{i+1}:=\frac{T_{i+1}}{(z_{1},X_{1}^{(0)},\ldots,X_{1}^{(i)})T_{i+1}}.
$$
Now it follows that a new sequence:
$$
\begin{CD}
\overline{\mathcal{T}}:R^{+}/pR^{+}=\overline{T}_{0} @>>> \overline{T}_{1} @>>> \cdots @>>> \overline{T}_{s} \\
\end{CD}
$$ 
is a sequence of modifications of $R^{+}/pR^{+}$ with respect to $z_{2},\ldots,z_{d}$ satisfying $1 \in (z_{2},\ldots,z_{d})\overline{T}_{s}$. Furthermore, if $Q$ is a minimal prime ideal of $R^{+}$ over $pR^{+}$, then $R^{+}/Q$ is the absolute integral closure of $R/pR$. We then replace $\overline{\mathcal{T}}$ with $\overline{\mathcal{T}} \otimes (R^{+}/Q)$ and get a sequence of bad modifications of $R^{+}/Q$ over $R/pR$. However, since $R^{+}/Q$ is a big Cohen-Macaulay algebra over $R/pR$ by Theorem~\ref{HHHL}, we apply a standard technique (see~\cite{HH2}, Proposition 3.3, or~\cite{Ho94}, Proof of Theorem 11.1) to construct the following commutative diagram:
$$
\begin{CD}
R^{+}/Q @= R^{+}/Q @= \cdots @= R^{+}/Q \\
@AAA @AAA @. @AAA \\
R^{+}/Q @>>> \overline{T}_{1} @>>> \cdots @>>> \overline{T}_{s} \\
\end{CD}
$$
in which the first vertical arrow is the identity map, and so a contradiction: $(z_{2},\ldots,z_{d})R^{+}/Q=R^{+}/Q$. Hence we have proved $(1)$ and $(2)$.

To show that $z_{1}$ is not nilpotent in $T$, we may construct a similar commutative diagram as above. By the preceding discussion, $z_{1},\ldots,z_{d}$ forms an improper regular sequence on the localized algebra $\mathbf{E}(R^{+})_{z_{2} \ldots z_{d}}$. Suppose that there is a sequence of modifications of type $\ge 1$:
$$
\begin{CD}
\mathcal{T}: \mathbf{E}(R^{+})=T_{0} @>>> T_{1} @>>> \cdots @>>> T_{s} \\
\end{CD}
$$
such that $z_{1}$ is nilpotent in $T_{s}$. Then we may construct a commutative diagram:
$$
\begin{CD}
\mathbf{E}(R^{+})_{z_{2} \cdots z_{d}} @= \mathbf{E}(R^{+})_{z_{2} \cdots z_{d}} @= \ldots @= \mathbf{E}(R^{+})_{z_{2} \cdots z_{d}} \\
@AAA @AAA @. @AAA \\
\mathbf{E}(R^{+}) @>>> T_{1} @>>> \cdots @>>> T_{s} \\
\end{CD}
$$ 
However, this diagram clearly contradicts that $z_{1}$ is nilpotent in $T_{s}$.  Denote by $C$ the direct limit of various sequences of algebra modifications of type $\ge 1$ over $\mathbf{E}(R^{+})$ with respect to $z_1,z_2,\ldots,z_d$. Then $z_{1}$ is not nilpotent in $C$ as well. Hence this proves $(3)$. From the above construction, there is a natural $\mathbf{E}(R^{+})$-algebra homomorphism $C \to \mathbf{E}(R^{+})_{z_{2}\cdots z_{d}}$.

To finish the proof, let us construct the desired algebra $S$ out of $C$. Since a perfect ring remains perfect under localization (\cite{Die}, Lemma 3.5), we take the perfect closure of the map $C \to \mathbf{E}(R^{+})_{z_{2}\cdots z_{d}}$ and obtain a commutative diagram of $\mathbf{E}(R^{+})$-algebra homomorphisms: 
$$
\begin{CD}
C @>>> \mathbf{E}(R^{+})_{z_{2}\cdots z_{d}} \\
@VVV @VVV \\
S @>>> \mathbf{E}(R^{+})_{z_{2}\cdots z_{d}} \\
\end{CD}
$$
Since $z_{1}$ is a nonzero divisor of $\mathbf{E}(R^{+})_{z_{2}\cdots z_{d}}$, $z_{1}$ is not nilpotent in $S$ as well, and $z_{2},\ldots,z_{d}$ forms a regular sequence on $S/z_{1}S$. Indeed for the latter assertion, let $J_{k}:=(z_{1},\ldots,z_{k})$ for $1 \le k \le d$. Then $z_{2}^{p^e},\ldots,z_{d}^{p^e}$ is a regular sequence on $C/z_{1}^{p^e}C$ and for $k \ge 1$, $S/J_{k}S$ is identified with the direct limit defined by the Frobenius map:
$$
\begin{CD}
C/J_{k}C @>\mathbf{F}>> C/J_{k}^{[p]}C @>\mathbf{F}>> C/J_{k}^{[p^2]}C @>\mathbf{F}>> \cdots. 
\end{CD}
$$ 
So we get our assertion. It remains to show that $z_{1}^{\epsilon} \cdot (0:_{S} z_{1})=0$ for any rational $\epsilon > 0$. Let $N:=(0:_{S} z_{1})$. Since the Frobenius map $\mathbf{F}_{S}:S \to S$ is bijective and there is an inclusion $N \subseteq \mathbf{F}^{-k}_{S}(N)$, we may iterate the Frobenius and get an injective sequence of ideals in $S$ ($\mathbf{F}^{0}_{S}$ is the identity map):
$$
\begin{CD}
z_{1}^{p^{-k}} \cdot N @>\hookrightarrow>>  z_{1}^{p^{-k}} \cdot \mathbf{F}^{-k}_{S}(N)  @>\hookrightarrow>> \cdots @>\hookrightarrow>> z_{1} \cdot N=0,
\end{CD}
$$
and as this holds for arbitrarily large $k > 0$, we complete the proof of the theorem.
\end{proof}

We are now ready to prove the main theorem.

\begin{Theorem}
\label{AlmostCM}
Let $(R,\fm)$ be a complete local domain of mixed characteristic $p > 0$. Then there exist a system of parameters $p,x_{2},\ldots,x_{d}$ of $R$ and a weakly almost Cohen-Macaulay $R^+$-algebra $B$ satisfying the following conditions:

\begin{enumerate}
\item[$\mathrm{(1)}$]
$(p,x_{2},\ldots,x_{d})B \ne B$;

\item[$\mathrm{(2)}$]
$x_{2},\ldots,x_{d}$ forms a regular sequence on $B/pB$; 

\item[$\mathrm{(3)}$]
$p$ is not nilpotent in $B$ and the ideal $(0:_{B} p)$ is annihilated by $p^{\epsilon}$ for any rational $\epsilon > 0$.

\end{enumerate}
\end{Theorem}

\begin{proof}
By Cohen's structure theorem, there exists an unramified complete regular local ring $A$ for which $A \to R$ is module-finite. By enlarging the residue field of each ring of $A \to R$ to its perfect closure, completing, and killing it by some minimal prime, we get a map $A' \to R'$ of complete local domains with perfect residue fields. Then it follows from (\cite{Mat}, Theorem 8.4) that $A' \to R'$ is module-finite. In other words, we are in the hypotheses of Theorem~\ref{Main1}, so that we may choose $p,x_{2},\ldots,x_{d}$ as a regular system of parameters of $A'$ by keeping track of the image of the regular system of parameters of $A$. We fix the notation as in Theorem~\ref{Main1}.

Recall that the $\mathbf{E}(R^{+})$-algebra homomorphism $S \to \mathbf{E}(R^{+})_{z_{2} \cdots z_{d}}$ has been constructed in the previous theorem. Taking their Witt rings, we have an $W(\mathbf{E}(R^{+}))$-algebra homomorphism:
$$
\begin{CD}
W(S) @>>> W(\mathbf{E}(R^{+})_{z_{2} \cdots z_{d}}), 
\end{CD}
$$
in which $p$ is a nonzero divisor, because both $S$ and $\mathbf{E}(R^{+})_{z_{2}\cdots z_{d}}$ are perfect algebras. But then Proposition \ref{Teichmuller} together with its following remark provides us a sequence of $R^{+}$-algebra homomorphisms:
$$
R^+ \to \widehat{R^+} \simeq \frac{W(\mathbf{E}(R^{+}))}{\vartheta \cdot W(\mathbf{E}(R^{+}))} \to B:=\frac{W(S)}{\vartheta \cdot W(S)} \to C:=\frac{W(\mathbf{E}(R^{+})_{z_{2} \cdots z_{d}})}{\vartheta \cdot W(\mathbf{E}(R^{+})_{z_{2} \cdots z_{d}})}
$$
for $\vartheta=\theta_{\mathbf{E}(R^{+})}(\langle p \rangle)-p$. To simplify notation, let us write $\langle x \rangle$ for $\theta_{\mathbf{E}(R^{+})}(\langle x \rangle)$. Then we have $\psi(\langle x_{i} \rangle)=x_{i}$ by Lemma~\ref{Witt}. Since $W(\mathbf{E}(R^{+})_{z_{2} \cdots z_{d}})$ is $p$-adically separated, one can easily verify that $\vartheta,p$ forms a regular sequence on $W(\mathbf{E}(R^{+})_{z_{2} \cdots z_{d}})$, and thus $p$ is a nonzero divisor of $C$ and $p$ is not nilpotent in $B$. It is obvious that $x_{2},\ldots,x_{d}$ is a regular sequence on $B/pB$. 

It remains to show that the ideal $(0:_{B} p)$ is annihilated by $p^{\epsilon}$ for any rational $\epsilon > 0$. Now assume $p x=(\langle p \rangle-p)y$ for $x,y \in W(S)$. Then $p(x+y)=\langle p \rangle y$. Since $S \simeq W(S)/p \cdot W(S)$, we have $\langle p \rangle^{\epsilon} y=py'$ for some $y' \in W(S)$ by Theorem~\ref{Main1} and thus,
$$
\langle p \rangle^{\epsilon}p(x+y)=\langle p \rangle^{\epsilon}\langle p \rangle y.
$$
Then this yields $\langle p \rangle^{\epsilon}px+p^{2}y'=\langle p \rangle p y'$, or $\langle p \rangle^{\epsilon}x+py'=\langle p \rangle y'$. Hence $\langle p \rangle^{\epsilon}x=(\langle p \rangle-p)y'$. On the other hand, $p$ and $\langle p \rangle$ become identical after applying the map $\psi:W(\mathbf{E}(R^{+})) \to \widehat{R^{+}}$ from Proposition~\ref{Teichmuller}. Thus, we conclude that $p^{\epsilon} x=0$ in $B$, which is our assertion. The desired almost Cohen-Macaulay $R^{+}$-algebra $B$ has successfully been constructed.
\end{proof}

\section{Concluding remarks}

In this final section, we show under some additional assumptions, that the existence of weakly almost Cohen-Macaulay algebras constructed previously yields the Monomial Conjecture in mixed characteristic. The Monomial Conjecture states that $x_1^t \cdots x_d^t \notin (x_1^{t+1},\ldots,x_d^{t+1})$ for all $t \in \mathbb{N}$ and all systems of parameters $x_1,\ldots,x_d$ of any local Noetherian ring. For the proof of the following corollary, we use Hochster's partial algebra modifications. Especially, we note that the nonzero divisor $c \in T$ which appears in (\cite{Ho02}, Lemma 5.1) can be replaced with a non-nilpotent element.

\begin{Corollary}
Let the $R^{+}$-algebra $B$ be the same as in Theorem~\ref{AlmostCM} and assume that 
$$
p^{\epsilon} \notin (p,x_{2},\ldots,x_{d})B
$$ 
for some rational $\epsilon > 0$. Then $B$ maps to a big Cohen-Macaulay $R^+$-algebra as an $R^+$-algebra. In particular, the Monomial Conjecture holds for any ring $T$ with $R \subseteq T \subseteq R^{+}$.
\end{Corollary}

\begin{proof}
It is a well-known fact that the Monomial Conjecture holds for any such subring $T$, if there exists a big Cohen-Macaulay algebra over $R^{+}$ (see \cite{Ho75} for example). We need to consider the sequences of modifications over the $R^{+}$-algebra $B$. If we end up with a bad sequence of modifications, there is a commutative diagram:
$$
\begin{CD} 
B[p^{-1}] @= B[p^{-1}] @= \cdots @= B[p^{-1}] \\
@AAA @AAA @. @AAA \\
B @>>> T_{1} @>>> \cdots @>>> T_{s} \\
\end{CD} 
$$
in which we have, as stated in (\cite{Ho02}, Theorem 5.2), that the leftmost vertical arrow is the natural map, and the image of each $T_{i}$ is contained in the cyclic module $p^{-\epsilon N_{i}} B$ for $0 \le i \le s$, arbitrarily small $\epsilon>0$ and some integer $N_{i} > 0$. Now we have $1 \in (p,x_{2},\ldots,x_{d}) p^{-\epsilon N}B$ for a fixed integer $N > 0$. Then this is just $p^{\epsilon N} \in (p,x_{2},\ldots,x_{d})B$. If we replace $\epsilon$ with $\epsilon N^{-1}$, we have $p^{\epsilon} \in (p,x_{2},\ldots,x_{d})B$, which is a contradiction to the hypothesis of the corollary. Hence $B$ maps to a big Cohen-Macaulay $R^+$-algebra.
\end{proof}

We discuss how strong our extra condition that $p^{\epsilon} \notin (p,x_{2},\ldots,x_{d})B$ would be. Assume that $B$ is already a big Cohen-Macaulay algebra over $(R,\fm)$. Then it is known that the $\fm$-adic completion of $B$ is ``balanced'' in the sense that every system of parameters for $R$ becomes a regular sequence on it. In this sense, no matter how huge the algebra $B$ may be, the completed algeba $\widehat{B}$ satisfies even the stronger condition that $0=\bigcap_{n>0} (p,x_2\ldots,x_d)^n\widehat{B}$. So our condition is at least necessary, if we require that $B$ maps to a big Cohen-Macaulay algebra. The above corollary also explains why Roberts assumes that $B/\fm B$ is not almost zero. In fact, his condition states that $\fm B$ does not contain elements of $R^+$ with arbitrarily small valuations.

For non-finitely generated modules over Noetherian rings, the separatedness condition is quite subtle. Here is a simple example.

\begin{Example}
Let $V$ be a complete discrete valuation ring, let $K$ be its field of fractions, and let $N:=V \oplus K$ as a $V$-module. If $t$ is any nonzero element in the maximal ideal of $V$, then it is a nonzero divisor of $N$, but we have $t^n N=t^n V \oplus K$, because $t^n K=K$. It follows that the $t$-adic completion of $N$ is just $V$ and thus, $N$ is far from being separated.
\end{Example}

To end this section, we suggest a possible approach to the separatedness issue. Recall that Roberts' condition is that $B/\fm B$ is not almost zero as an $R^+$-module. If one looks at $B$ as a module over itself, the trouble with $B$ is that it may not be a domain, so that one cannot define a valuation on it directly. However, by finding a ``pseudo'' valuation on $B$, one can obtain the same conclusion as in the above corollary. To make it clear, we use the convention that $\infty=0 \cdot \infty$. The argument in the next corollary is found by Asgharzadeh \cite{ASh}.

\begin{Corollary}
Let the $R^{+}$-algebra $B$ be the same as in Theorem~\ref{AlmostCM} and assume that there is a function $v:B \to \mathbb{R} \cup \{\infty\}$ satisfying the following conditions:

\begin{enumerate}
\item[$\mathrm{(1)}$]
\mbox{$v(ab)=v(a)+v(b)$ for all $a, b \in B$};

\item[$\mathrm{(2)}$]
\mbox{$v(a+b) \ge \min\{v(a), v(b)\}$ for all $a, b \in B$};

\item[$\mathrm{(3)}$]
$v(0)=\infty$;

\item[$\mathrm{(4)}$]
\mbox{$v$ is non-negative on $B$ and $v(b)>0$ for any nonunit element $b \in B$}.
\end{enumerate}
Then $B$ maps to a big Cohen-Macaulay $R^+$-algebra as an $R^+$-algebra.
\end{Corollary}

In the proof, we use almost zero modules with respect to $v:B \to \mathbb{R} \cup \{\infty\}$. This is found in \cite{ASh}. But we only need to keep in mind that the definition of almost zero modules over an algebra (which is not necessarily a domain) is in the same format as given in \cite{Ro2}.

\begin{proof}
First, we prove that $B/\fm B$ is not almost zero with respect to the function $v$. For a contradiction, suppose that $B/\fm B$ is almost zero. In particular, $1 \in B/\fm B$ is almost zero. Then for any $\epsilon>0$, we can find $b \in B$ such that $v(b)<\epsilon$ and $b \in \fm B$, which says that $\fm B$ has elements with arbitrarily small valuations. For any $a \in \fm B$, writing $a=\sum_{i=1}^{n} a_iz_i$ with $\fm=(z_1,\ldots,z_n)$, we find that
$$
v(a) \ge \min\{v(a_iz_i)~|~1 \le i \le n\}=\min\{v(a_i)+v(z_i)~|~1 \le i \le n\}
 \ge \min\{v(z_i)~|~1 \le i \le n\},
$$
which implies that $v(a)$ is bounded from below by some positive constant. But this is a contradiction and so $B/\fm B$ is not almost zero.

On the other hand, for $\epsilon:=\frac{m}{n}>0$ with positive integers $m, n$, we have $v(p^{\epsilon})=\epsilon \cdot v(p)>0$, because $p$ is not a unit in $B$.
This shows that we cannot have $p^{\epsilon} \notin (p,x_{2},\ldots,x_{d})B$. So the conclusion follows from the previous corollary.
\end{proof}

In addition to the conditions of the corollary, assume that $b=0 \iff v(b)=\infty$. Then such a map is a valuation on $B$ and thus $B$ is a domain. This is why we assumed only $b=0 \Longrightarrow v(b)=\infty$ in the corollary. Alternatively, one may formulate both of the above corollaries over the Fontaine ring $\mathbf{E}(R^+)$ rather than over $R^+$, so that the conclusion would be to say that the $\mathbf{E}(R^+)$-algebra $S$ in Theorem \ref{Main1}  maps to an algebra, in which $\langle p \rangle,\langle x_2 \rangle,\ldots,\langle x_d \rangle$ becomes a regular sequence. As $S$ is obtained as the perfect closure of the huge direct limit of various algebra modifications over $\mathbf{E}(R^+)$, one could possibly pave a way to find a function $v:S \to \mathbb{R} \cup \{\infty\}$ satisfying the required conditions.

We also would like to point out that, although $R^+$ is not Noetherian, it is $\fm$-adically separated~\cite{Ho02}. By this fact, it seems that some condition similar to $p^{\epsilon} \notin (p,x_{2},\ldots,x_{d})B$ can be satisfied. So finally, let me simply say that our results strongly uphold the Monomial Conjecture in the mixed characteristic case.


\begin{thebibliography}{99}

     
\bibitem{Ar} 
M. Artin, \emph{On the joins of Hensel rings}, Advances in Math. \textbf{7}  (1971),  282--296.

\bibitem{ASh}
M. Asgharzadeh and K. Shimomoto, \emph{Almost Cohen-Macaulay and almost regular algebras via almost flat extensions}, submitted.

\bibitem{Die}
G. D. Dietz,  \emph{Big Cohen-Macaulay algebras and seeds}, Trans. Amer. Math. Soc. \textbf{359} (2007), 5959--5989.

\bibitem{FonOuy}
J.-M. Fontaine and Y. Ouyang, \emph{Theory of $p$-adic Galois representations}, forthcoming book, Springer-Verlag. 

\bibitem{GaRa}
O. Gabber and L. Ramero, \emph{Foundations of $p$-adic Hodge theory}, \textbf{arXiv:math/0409584}.

\bibitem{Hei}
R. Heitmann, \emph{The direct summand conjecture in dimension three}, Annals of Math. \textbf{156} (2002), 695--712.
 
\bibitem{Ho75}
M. Hochster,  \emph{Topics in the homological theory of modules over commutative rings}, CBMS Regional conference series  \textbf{24}, Amer. Math. Soc. (1975).

\bibitem{Ho94}
M. Hochster,  \emph{Solid closure}, Contemp. Math. \textbf{159} (1994), 103--172.

\bibitem{Ho02}
M. Hochster,  \emph{Big Cohen-Macaulay algebras in dimension three via Heitmann's theorem},  J. of Algebra \textbf{254} (2002), 395--408.

\bibitem{HH1}
M. Hochster and C. Huneke, \emph{Infinite integral extensions and big Cohen--Macaulay algebras}, Annals of Math. \textbf{135} (1992), 53--89.

\bibitem{HH2}
M. Hochster and C. Huneke, \emph{Applications of the existence of big Cohen--Macaulay algebras}, Advances in Math. \textbf{113} (1995), 45--117.
 
\bibitem{HL}
C. Huneke and G. Lyubeznik,  \emph{Absolute integral closure in positive characteristic}, Advances in Math. \textbf{210} (2007), 498--504.

\bibitem{Mat}
H. Matsumura, \emph{Commutative ring theory}, Cambridge University Press. \textbf{8} (1986).
  
\bibitem{Ro1}
P. Roberts, \emph{The root closure of a ring of mixed characteristic},  \textbf{arXiv:0810.0215}.

\bibitem{Ro2}
P. Roberts, \emph{Fontaine rings and local cohomology},  J. Algebra \textbf{323} (2010), 2257--2269.

\bibitem{Se}
J.-P. Serre, \emph{Local fields}, Graduate Texts in Mathematics 
\textbf{67}, Springer-Verlag, (1979).
  

\end{thebibliography}
\end{document}